\documentclass[12pt,reqno]{amsart}

\usepackage{amssymb, a4wide, amsmath,amsfonts,latexsym, mathtools,xy}
\usepackage[T1]{fontenc}
\usepackage[utf8]{inputenc}
\usepackage{graphicx}
\usepackage{xspace}
\xyoption{all}

\usepackage[usenames]{xcolor}

\makeatletter
\DeclareFontEncoding{LS2}{}{\@noaccents}
\makeatother
\DeclareFontSubstitution{LS2}{stix}{m}{n}

\DeclareSymbolFont{largesymbolsstix}{LS2}{stixex}{m}{n}

\DeclareMathDelimiter{\lBrace}{\mathopen}{largesymbolsstix}{"E8}{largesymbolsstix}{"0E}
\DeclareMathDelimiter{\rBrace}{\mathclose}{largesymbolsstix}{"E9}{largesymbolsstix}{"0F}

\usepackage{hyperref}

\newtheorem{definition}{Definition}[section]
\newtheorem{theorem}[definition]{Theorem}
\newtheorem{proposition}[definition]{Proposition}

\newtheorem{lemma}[definition]{Lemma}
\newtheorem{example}[definition]{Example}
\newtheorem{remark}[definition]{Remark}

\def\lra{\longrightarrow}

\DeclareMathOperator{\A}{\mathsf {A}}
\DeclareMathOperator{\B}{\mathsf {B}}
\DeclareMathOperator{\C}{\mathsf {C}}
\DeclareMathOperator{\F}{\mathsf {F}}
\DeclareMathOperator{\Z}{\mathcal{Z}}
\DeclareMathOperator{\Vect}{\mathsf{Vect}}
\DeclareMathOperator{\AWB}{\mathsf{AWB}}

\DeclareMathOperator{\TAWB}{\mathsf{TAWB}}
\DeclareMathOperator{\As}{\mathsf{Ass}}
\DeclareMathOperator{\XAWB}{\mathsf{XAWB}}
\DeclareMathOperator{\ab}{\mathrm{ab}}
\DeclareMathOperator{\as}{\mathrm{ass}}
\DeclareMathOperator{\M}{\mathsf {M}}
\DeclareMathOperator{\N}{\mathsf {N}}

\newcommand{\awb} {{\footnotesize \textup{AWB}}\xspace}
\newcommand{\awbs} {{\footnotesize {AWB}}s\xspace} 
\newcommand{\Hoch}{\textup{Hoch}}

\newcommand{\re}{\mathsf R}
\newcommand{\fe}{\mathsf F}
\newcommand{\pe}{\mathsf P}

\newcommand{\K}{\mathbb{K}}
\newcommand{\Ker}{\mathsf{Ker}}
\newcommand{\Coker}{\mathsf{Coker}}
\newcommand{\Ima}{\mathsf{Im}}
\newcommand{\id}{\mathsf{Id}}
\newcommand{\inc}{\mathsf{inc}}

\title[A non-abelian tensor product  of algebras with
bracket]{A non-abelian tensor product  of algebras with
	bracket}

\author[J. M. Casas]{Jos\'e Manuel Casas\textsuperscript{1}}
\address{\textsuperscript{1}Departamento de Matem\'atica Aplicada I \& CITMAga, E. E. Forestal,  University of Vigo, 36005 Pontevedra, Spain}
\email{jmcasas@uvigo.es \ (ORCID: 0000-0002-6556-6131)}

\author[E. Khmaladze]{Emzar Khmaladze\textsuperscript{2}}
\address{\textsuperscript{2}The University of Georgia, Kostava St. 77a, 0171 Tbilisi, Georgia \& A. Razmadze Mathematical Institute of Tbilisi State University,
	Tamarashvili St. 6, 0177 Tbilisi, Georgia}
\email{e.khmaladze@ug.edu.ge \ (ORCID: 0000-0001-9492-982X)}

\author[M. Ladra]{Manuel Ladra\textsuperscript{3}}
\address{\textsuperscript{3}Departamento de Matem\'aticas \& CITMAga, Universidade de Santiago de Compostela, 15782 Santiago de Compostela, Spain}
\email{manuel.ladra@usc.es \ (ORCID: 0000-0002-0543-4508)}



\numberwithin{equation}{section}

\begin{document}

\begin{abstract}
	We introduce and study a non-abelian tensor product of  two algebras with bracket with compatible actions on each other. We investigate its applications to the universal central extensions and the low-dimensional homology of perfect algebras with bracket. 
\end{abstract}

\maketitle
\section{Introduction}

The non-abelian tensor product in various algebraic categories, such  as groups, (Hom)-Lie and (Hom)-Leibniz algebras, Lie superalgebras. etc., plays an essential role in the study of homotopy theory, low-dimensional homology or universal central extensions (see \cite{BL, CKP, El2, El1, GKL, Gu1, In1, In2}). It has also been used in the construction of low-dimensional non-abelian homology of groups, Lie algebras  and Leibniz algebras, having interesting applications to the algebraic $K$-theory, cyclic homology and Hochschild homology, respectively \cite{Gu1, Gu2, Gn}. 



In this paper, we choose to develop a non-abelian tensor product for a relatively new algebraic structure
called algebra with bracket, introduced  in~\cite{CP} as a kind of generalization of the (non-commutative) Poisson algebra. It should be noted here that such a generalization of Poisson algebras originates in physics literature (see, e.g.~\cite{Ka}). 
Among other results in ~\cite{CP},  Quillen cohomology of algebras with bracket is described via an explicit cochain complex. 
Further (co)homological investigations of algebras with bracket are carried out in \cite{Ca, CKL}. 
In particular, for our importance, we mention that in \cite{Ca}, a homology with trivial coefficients of algebras with bracket is developed with applications to universal central extensions. 
In \cite{CKL}, crossed modules for algebras with bracket are introduced, and the second cohomology is interpreted as the set of equivalence classes of crossed extensions. The eight-term exact cohomology sequence is also constructed.

In the present paper, we continue the same line of homological study of algebras with bracket. We fit the homology with trivial coefficients \cite{Ca} into the context of Quillen homology, introduce the non-abelian tensor product of algebras with bracket by generators and relations, and give applications in universal central extensions and low-dimensional homology.


The organization of this paper is as follows: after the introduction, in Section~\ref{S:AWB}, we present some definitions and results for the development of the paper.
 We briefly recall the construction of homology for algebras with bracket from~\cite{Ca, CP} and prove that it is consistent with the context of 
 Quillen's homology theory (Theorem~\ref{T:Quillen}). In Section~\ref{S:CMAWB}, we present all the ingredients for developing the non-abelian tensor product later.
  In particular, we define actions,  semi-direct products and crossed modules of algebras with bracket. 
  Additionally, we show that the category of crossed modules is equivalent to the category of cat$^1$-algebras with bracket (Theorem~\ref{equiv}).
   Section~\ref{non-ab} contains the main results of the paper. Here, we present the construction of the non-abelian tensor product of
    two algebras with bracket acting compatibly on each other (Proposition~\ref{structure}) and  study its properties.
     Regarding trivial actions, we  describe  the non-abelian tensor product (Proposition~\ref{P:trivial_action}). 
      We  establish a right-exactness property of the non-abelian tensor product of algebras with bracket (Theorem~\ref{T:right_exactness}) and 
  equip the non-abelian tensor product with crossed module structures (Proposition~\ref{maps_phi_psi}).
     Finally,  as an application, for a given perfect algebra with bracket, we construct its universal central extension
 (Theorem~\ref{T:awb_uce}) and  a four-term exact homology sequence (Theorem~\ref{T:4_term}).

\section{ Algebras With Bracket} \label{S:AWB}
Throughout the paper we  fix a ground  field $\mathbb{K}$. All vector spaces and algebras are $\mathbb{K}$-vector spaces and $\mathbb{K}$-algebras, and linear maps are $\mathbb{K}$-linear maps as well. In what follows $\otimes$ means $\otimes_{\mathbb{K}}$.

 \subsection{Basic definitions} \label{Basic}
 
\begin{definition}[\cite{CP}] An  algebra with bracket, or an \awb  for short,  is  an associative (not necessarily commutative) algebra ${\A}$
	equipped with a bilinear map (bracket operation) $[-,-] \colon  {\A} \times {\A} \to {\A}$, $(a, b)\mapsto [a,b]$ satisfying the following identity:
	\begin{equation}\label{awbequa}
		[a  b,c] = [a,c]  b + a  [b,c]
	\end{equation}
	for all $a, b, c \in {\A}$.
\end{definition}

\emph{A homomorphism} of \awbs is a homomorphism of associative algebras preserving the bracket operation.
We denote by $\AWB$ the respective category of \awbs.

The category $\AWB$ is a variety of $\Omega$-groups \cite{Hig}, and therefore it is a semi-abelian category~\cite{BB, JMT}: pointed, exact and protomodular with binary coproducts. So classical lemmas such as the Five Lemma \cite{Mi} hold for \awbs which we will use later on.
Below we give the definitions of ideal, center,
		commutator, action and semi-direct product of \awbs, and of course
		these notions agree with the corresponding general notions in the context of semi-abelian
		categories.
		

 We now list some common examples of \awbs that will be discussed later. Other examples can be found in \cite{Ca1, CKL, CP, Ka}.

\begin{example}\label{ejemplos}\
	
	\begin{itemize}
		\item[(i)]  Any vector space $\A$ enriched with the trivial multiplication and bracket operation, 
		i.e. $ab=0$ and $[a,b]=0$ for all $a,b\in \A$,  is an \awb, called an abelian \awb.
		 Hence, the category of vector spaces is a full subcategory of $\AWB$ and the respective inclusion functor $ \Vect \hookrightarrow \AWB$ has a left adjoint, the so called abelianization functor  $(-)^{\ab} \colon \AWB \to  \Vect$, which will be described  in Remark~\ref{remark ab as} (i) below.
		
			\item[(ii)]  Any associative algebra $\A$ together with the trivial bracket operation can be regarded as an \awb.  This defines the inclusion functor $I \colon \As \to \AWB$, where $\As$ denotes the category of associative algebras. The functor $I$ has a left adjoint $(-)^{\as} \colon \AWB\to \As$ described in Remark~\ref{remark ab as} (ii) below.
			
			\item[(iii)] Another way of considering an associative algebra $\A$ as an \awb is to define the bracket operation by
			\[
			[a,b]\coloneqq ab - ba, \quad a, b\in \A.
			\]
			This particular \awb is called the \emph{tautological} \awb associated to the associative algebra ${\A}$ and will be denoted by $T(\A)$.
			
		Tautological \awbs constitute a full subcategory of $\AWB$ denoted by $\TAWB$. The correspondence $T \colon \As \to \TAWB $, $	{\A} \mapsto T(\A)$, is functorial, and it establishes an isomorphism between the categories  $\TAWB$ and $\As$.

		\item[(iv)]  Any Poisson algebra is an \awb. In fact, the category $\mathsf {Poiss}$ of (non-commutative) Poisson algebras is a subcategory of $\AWB$. The inclusion functor $\mathsf {Poiss} \hookrightarrow \AWB$ has as left adjoint the functor given by
		${\A} \mapsto {\A}_{\mathsf {Poiss}}$, where ${\A}_{\mathsf {Poiss}}$ is the maximal quotient of ${\A}$, such that the following relations hold:  $[a,a]\sim 0$ and $[a,[b,c]]+[b,[c,a]]+[c,[a,b]] \sim 0$.

	\end{itemize}
\end{example}

 The following notions for \awbs are given in \cite{Ca} and they agree with the corresponding  notions in semi-abelian categories. 
\emph{A subalgebra} ${\B}$ of an \awb ${\A}$ is a vector subspace  which is closed under the product and the bracket operation, that is, ${\B}~{\B} \subseteq
{\B}$ and $[{\B},{\B}] \subseteq {\B}$.
A subalgebra ${\B}$ is said to be a  \emph{left (respectively, right) ideal} if  ${\A}~{\B} \subseteq {\B}$, $[{\A},{\B}] \subseteq {\B}$ (respectively,
${\B}~{\A} \subseteq {\B}$, $[{\B},{\A}] \subseteq {\B}$).
If  ${\B}$ is both left and right ideal, then it is said to be  a \emph{two-sided ideal}. In this case, the quotient
${\A}/{\B}$ is endowed with an \awb structure  naturally induced from the operations on ${\A}$.

Let $\A$ be an \awb and ${\B}, {\C}$ be two-sided ideals of ${\A}$. The \emph{commutator ideal} of ${\B}$ and ${\C}$ is the two-sided ideal of $\A$ 
\[
[[{\B},{\C}]] = \langle \{bc, cb, [b,c], [c,b] \mid b \in {\B}, c \in {\C}
\} \rangle.
\]
Obviously $[[{\B}, {\C}]] \subseteq {\B} \bigcap {\C}$. In the particular case $ {\B} = {\C} = {\A}$, one obtains the definition of derived algebra of ${\A}$,
\[
[[{\A},{\A}]]=\langle \{aa', [a,a'] \mid a, a' \in {\A} \} \rangle.
\]
An \awb ${\A}$ is said to be \emph{perfect} if ${\A} =[[{\A},{\A}]]$.


	\begin{remark}\label{remark ab as} \hfill
\begin{itemize} 
\item[(i)] Given an \awb $\A$, the quotient ${\A}/[[{\A},{\A}]] $ is always an abelian \awb and will be denoted by $\A^{\ab}$.
	Abelian \awbs (i.e. just vector spaces) are abelian group objects in the category $\AWB$. The respective abelianization functor $(-)^{\ab} \colon \AWB \to  \Vect$, which is left adjoint to the inclusion functor $ \Vect \hookrightarrow \AWB$, sends an \awb $\A$ into $\A^{\ab}={\A}/[[{\A},{\A}]]$.
	\item[(ii)] To an \awb $\A$, we associate the associative algebra  
	\[ 
	\A^{\as}= \A / [\A,\A], \quad \text{where} \quad [{\A},{\A}]=\langle \{ [a,a'] \mid a, a' \in {\A} \} \rangle .
	\] 
	This correspondence is functorial and satisfies the following universal property: given an associative algebra $\B$, any homomorphism of \awbs $\A \to I(\B)$ factors trough $\A^{\as}$, where $I \colon \As\hookrightarrow \AWB$ is the inclusion functor as in Example~\ref{ejemplos} (iii). Thus, the functor $(-)^{\as} \colon \AWB\to \As$, $\A\mapsto \A^{\as}$, is left adjoint to $I$.
	\end{itemize}
\end{remark}

The \emph{center} of an \awb ${\A}$ is the two-sided ideal
\[
{\Z}({\A}) =  \{a \in {\A} \mid ab
= 0 = ba,[a,b] = 0 = [b,a], \ \text{for  all}\ b \in {\A}\}.
\]
Note that an \awb ${\A}$ is abelian if and only if ${\A} = {\Z}({\A})$.

\emph{A central extension} of an \awb $\A$ is an exact sequence of \awbs $0 \to {\M} \to {\B} \xrightarrow{\phi}  {\A} \to 0$  such that $[[{\M}, {\B}]] =0$ (equivalently, ${\M} \subseteq {\Z}({\B})$). It is said to be \emph{universal central extension} if for every central extension $0 \to {\N} \to {\C} \xrightarrow{\psi}  {\A} \to 0$ there is a unique homomorphism of \awbs $\alpha \colon {\B} \to {\C}$ such that $\psi \circ \alpha = \phi$. 

\begin{theorem}[\cite{Ca}] \label{uce} 
	An \awb ${\A}$ admits a universal central extension if and only if ${\A}$ is perfect. Moreover, the kernel of the universal central extension is isomorphic   to the first homology of $\A$, $\mathsf H_1^{\awb}(\A)$ (see the definition below).
\end{theorem}

\subsection{Homology}
In this subsection, we briefly review the homology of \awbs with trivial coefficients given in \cite{Ca, CP}.

Let $V$ be a vector space. Let ${R}_1(V)=V$ and ${R}_n(V)=V^{\otimes n}\oplus V^{\otimes n}$, if $n\geq 2$. In order to distinguish elements from these tensor powers, we let $a_1\otimes \cdots \otimes  a_n$ be a typical element from the first component  of
${R}_n(V)$, while $a_1\circ \cdots \circ a_n$ from the second component of ${R}_n(V)$.

Given an \awb  ${\A}$,  
we let $\left ( C_*^{\awb}({\A}) , d_* \right )$ be the chain complex defined by
\[C_n^{\awb}({\A}) \coloneqq R_{n+1}({\A}), \ n \geq 0,\]
with the boundary maps $d_n \colon C_n^{\awb}({\A})\to C_{n-1}^{\awb}({\A})$,  $n \geq 0$, given by
\begin{align*}
	& d_{n}(a_1\otimes \dots \otimes a_{n+1}) = \sum_{i=1}^{n} (-1)^{i+1}
	a_1 \otimes \dots \otimes a_i  a_{i+1} \otimes \dots \otimes a_{n+1}, \\
	& d_{n}(a_1\circ \cdots \circ a_{n+1}) = \sum_{i=1}^{n}
	a_1 \otimes \dots \otimes  [a_i , a_{n+1}] \otimes  \dots \otimes
	a_{n} + \sum_{i=1}^{n-1} (-1)^i a_1\circ \cdots \circ a_i a_{i+1} \circ \dots \circ a_{n+1}
\end{align*}

The homology of the complex $\left ( C_*^{\awb}({\A}) , d_* \right )$ is called the \emph{homology with trivial
	coefficients} of the \awb ${\A}$ and it is denoted  by  ${\mathsf H_*^{\awb}}({\A})$.

Easy computations show that there is an isomorphism
\[
{\mathsf H_0^{\awb}}({\A}) \cong {\A}/[[{\A},{\A}]].
\]
On the other hand, given a free presentation of $\A$, that is, a short exact sequence of \awbs $0 \to {\re} \to {\fe} \to {\A} \to 0$, where $\fe$ is a free \awb, then there is an isomorphism 
\[
{\mathsf H^{\awb}_1}({\A}) \cong ({\re} \cap [[{\fe},{\fe}]])/[[{\re},{\fe}]]
\] 
(see \cite[Corollary 2.14]{Ca}).


	\begin{remark}
	 Let $\A$ be an associative algebra and consider the ground field $\K$ as a trivial $\A$-bimodule. Let $C_*^{\Hoch}\!({\A}) = C_*^{\Hoch}\!(\A, \K)$ and ${\Hoch}_*\!({\A})={\Hoch}_*\!({\A}, \K)$ denote the Hochschild complex and the Hochschild homology of $\A$ with coefficients in $\K$ \cite{We}, respectively. Then
	 $C_{1}^{\Hoch}\!({\A})=\A=C_0^{\awb}(T(\A))$ and
	  the  natural injections
\[
C_{n+1}^{\Hoch}\!({\A}) = \A^{\otimes (n+1)} \hookrightarrow \A^{\otimes (n+1)}\oplus\A^{\otimes (n+1)}=C_n^{\awb}(T(\A)), \quad n\geq 1
\] 
gives rise to a morphism of chain complexes $C_{*+1}^{\Hoch}\!({\A}) \hookrightarrow C_*^{\awb}(T(\A))$. Thus, we have an induced homomorphism in homology ${\Hoch}_{n+1}\!({\A}) \to \mathsf H_n^{\awb}({T(\A)})$ ($n\geq 0$), which is clearly an epimorphism for $n= 1$ and  an isomorphism for $n=0$.
\end{remark}


Now we show that the homology of \awbs is fitted in the context of homology theory developed by Quillen in
a very general framework \cite{Qu}. Let us recall that the Quillen homology of an object in an algebraic category $\mathcal{C}$ is defined via
the derived functors of the abelianization functor $(-)^{\ab} \colon \mathcal{C} \to \mathcal{C}^{\ab} $ from $\mathcal{C}$ to the abelian category
		 $\mathcal{C}^{\ab}$ of abelian group objects in $\mathcal{C}$. To specify this theory for \awbs, we proceed as follows.

Given an \awb $\A$, choose any free simplicial resolution $\F_*$ of $\A$, that is, an aspherical augmented simplicial \awb $\F_* \xrightarrow{\epsilon} \A$ (which means that all non-zero homotopies are trivial, $\pi_n(\F_*)=0$ for $n\geq 0$, and $\epsilon$ induces an isomorphism $\pi_0(\F_*)\cong \A$) such that each component $\F_n$, $n \geq 0$, is a free \awb.
Then the $n$-th Quillen homology of $\A$ is defined by 
\[
H^{Q}_n(\A)=H_n (\F_*^{\ab}), \quad  n \geq 0.
\] 
Here $\F_*^{\ab}$ is the simplicial vector space obtained by applying the functor $(-)^{\ab}$ dimension-wise to $\F_*$.

In the proof of the theorem immediately below, we need to use the result from  \cite{CP} that if $\F$ is a free \awb, then the homology of the complex $\left ( C_*^{\awb}({\F}) , d_* \right )$ vanishes in positive dimensions, that is, 
	\[
	{\mathsf H_n^{\awb}}({\F})=0, \quad \text{for} \quad n\geq 1.
	\]

\begin{theorem}\label{T:Quillen}
Let $\A$ be an \awb. Then there is an isomorphism of vector spaces
\[
{\mathsf H^{\awb}_n}({\A}) \cong H^{Q}_n(\A), \quad n \geq 0.
\]
\end{theorem}
\begin{proof} 
	First of all let us note that the homology chain complex $C_{*}^{\awb}$ is functorial in the sense that a homomorphism $\A \to \A'$   gives rise the chain map $C_*^{\awb}(\A) \to C_*^{\awb}(\A')$ in the canonical way. 
	
Now, given a free simplicial resolution $\F_*$ of $\A$, by applying the functor $C_{n}^{\awb}$ dimension-wise, and then taking the alternating sums of face homomorphisms, we get an augmented chain complex of vector spaces  $C_{n}^{\awb}(\F_*) \to C_{n}^{\awb}(\A)$. Since $\F_* \to \A$ is an aspherical simplicial \awb, we claim that $C_{n}^{\awb}(\F_*) \to C_{n}^{\awb}(\A)$ is acyclic chain complex for any $n\geq 0$.  This is easy to confirm, since by forgetting \awb structure in the simplicial \awb $\F_* \to \A$, we get a simplicial vector space having a linear left (right) contraction.

Then using the facts that $\mathsf H^{\awb}_n({\F_m})=0$ and $\mathsf H^{\awb}_0({\F_m})=\F_m^{\ab}$ for any $n\geq 1$ and $m\geq 0$
it follows that both spectral sequences for the bicomplex $C_*^{\awb}(\F_*)$ degenerate and give the required isomorphism.
\end{proof}

\section{Crossed modules of \awbs}\label{S:CMAWB}
\subsection{Actions and semi-direct product}

\begin{definition} \label{action def} 
	Let ${\A}$ and ${\M}$ be two \awbs. An action of ${\A}$ on ${\M}$ consists of four bilinear maps
\[
\begin{array}{llcll}
	{\A} \times {\M} \to {\M}, & (a, m) \mapsto {^{a \cdot} m}, & & {\M} \times {\A}  \to {\M}, &(m, a) \mapsto m^{\cdot a},\\
	{\A} \times {\M} \to {\M}, & (a, m) \mapsto {^{a \ast} m}, & & {\M} \times {\A}  \to {\M}, & (m, a) \mapsto m^{ \ast a},\\
\end{array}
\]
such that the following conditions hold:
	\begin{align}\label{equations_action}
				{^{(a_1 a_2) \cdot}} m &= {^{a_1 \cdot}} \left({^{a_2 \cdot}} m\right), & \left(^{a_1 \cdot} m \right)^{ \ast a_2} &= {^{a_1 \cdot}} \left( m^{ \ast a_2} \right) + {^{[a_1, a_2] \cdot}} m, \notag\\
				m ^{\cdot {(a_1 a_2)}} &= \left(m^{ \cdot {a_1}}\right)^{ \cdot {a_2}} , & \left(m^{ \cdot {a_1}}\right)^{ \ast {a_2}} &=  \left(m^{ \ast {a_2}}\right)^{ \cdot {a_1}} + m^{ \cdot {[a_1, a_2]}},\notag\\
				\left({^{a_1 \cdot}} m \right)^{ \cdot {a_2}}  &= {^{a_1 \cdot}} \left(m^{ \cdot {a_2}} \right) , & {^{\left(a_1 a_2\right) \ast}} m &= {^{a_1 \cdot}} \left( {^{a_2 \ast}} m \right) + \left( {^{a_1 \ast}} m \right)^{ \cdot {a_2}},\\
				(m_1 m_2)^{ \cdot a} &= m_1 (m_2^{ \cdot {\ a}}), & [m_1^{ \cdot { a}},m_2] &= m_1 \left({^{a \ast}} {m_2}\right) + [m_1, m_2]^{ \cdot a},\notag \\
				{^{a \cdot}} \left(m_1 m_2\right)  &= \left({^{a \cdot}}  m_1 \right) m_2 , & [{^{a \cdot}} {m_1}, m_2] &=  {^{a \cdot}} [m_1, m_2] + \left({^{a \ast}} {m_2} \right) m_1,\notag\\
				\left(m_1^{ \cdot  a}\right) m_2 & = m_1 \left({^{a \cdot}} {m_2} \right) , & (m_1 m_2)^{ \ast a} &= {{m_1}}\left( {m_2}^{ \ast a} \right) + \left( {m_1}^{ \ast a} \right) m_2,\notag
	\end{align}
	for all $a, a_1, a_2 \in {\A}$, $m, m_1, m_2 \in {\M}$. The action is called trivial if all these bilinear maps are trivial.
\end{definition}

Let us remark that if an action of an \awb  $\A$ on an abelian \awb  $\M$ is given,
then all six equations in the last three lines of \eqref{equations_action} vanish and we get the definition of a \emph{representation} $\M$ of $\A$ (see \cite{CP}).

\begin{example}\label{action}\
	\begin{itemize}
		\item[(i)] If $\M$ is a representation of an \awb  $\A$ thought as an abelian \awb, then there is an action of $\A$ on the abelian \awb  $\M$.
		
		\item[(ii)] If $\A$ is a subalgebra of some \awb  $\B$
		(maybe $\A = \B$) and if $\M$ is a two-sided ideal of  $\B$, then the operations in ${\B}$ yield an action of ${\A}$ on ${\M}$  given by  ${^{a \cdot}m}= a m$, $m^{\cdot a} = m a$, ${^{a \ast} m} = [a,m]$, $m^{\ast a} = [m,a]$, for all $m \in {\M}$ and $a \in {\A}$.
		
		\item[(iii)] If $ 0\to {\M} \xrightarrow{i}  {\B} \xrightarrow{\pi}  {\A }\to 0$ is a split short exact sequence of \awbs, that is, there exists a homomorphism  $s\colon  {\A}\to {\B}$ of \awbs  such that $\pi \circ s = \id_{\A}$, then there is an action of  ${\A}$ on ${\M}$, given by:
		\begin{align*}
			{}^{a \cdot}m &= i^{-1}\big( s(a) i(m)\big), & m^{\cdot a} &= i^{-1}\big( i(m) s(a) \big), \\
			{}^{a \ast}m &= i^{-1}\big(\left[s(a), i(m) \right]\big), & m^{\ast a}&= i^{-1}\big(\left[i(m), s(a)\right]\big),
		\end{align*}
		 for any $a\in {\A}$, $m\in {\M}$.
		
		\item[(iv)] Any homomorphism of \awbs  $f \colon  {\A} \to {\M} $  induces an action of ${\A}$ on ${\M}$ in the
		standard way by taking images of elements of ${\A}$ and operations in  ${\M}$.
		
		\item[(v)] If $\mu \colon {\M}\to {\A}$ is a surjective homomorphism of \awbs and the kernel of $\mu$ is contained in the center of $\M$, i.e. $\Ker(\mu) \subseteq {\Z}({\M})$, then there is an action of ${\A}$ on ${\M}$, defined in the standard way, i.e. by choosing pre-images of elements of ${\A}$ and taking operations in ${\M}$.
	\end{itemize}
\end{example}

\begin{definition}
	Let ${\A}$ and ${\M}$ be \awbs with an action of ${\A}$ on ${\M}$. The semi-direct product of ${\M}$ and ${\A}$, denoted by ${\M} \rtimes {\A}$, is the \awb whose  underlying vector space is ${\M} \oplus {\A}$ endowed with the operations
	\begin{align*}
		(m_1, a_1) (m_2, a_2) & = \big(m_1 m_2 + {}^{a_1 \cdot} {m_2} +{m_1}^{\cdot a_2}, a_1 a_2 \big), \\
		[(m_1, a_1), (m_2, a_2)] & = \big([m_1, m_2] + {}^{a_1 \ast} {m_2} + {m_1}^{\ast a_2}, [a_1, a_2]\big)
	\end{align*}
	for all $m_1, m_2 \in {\M}$, $a_1, a_2 \in {\A}$.
\end{definition}

Given an action of an \awb $\A$ on $\M$, straightforward calculations show that the sequence of \awbs
\[
	0 \lra {\M} \xrightarrow{ \ i \ } {\M} \rtimes {\A} \xrightarrow{ \ \pi \ }  {\A} \lra 0
\]
where $i(m)=(m, 0), \pi(m,a)=a$, is exact. Moreover ${\M}$ is a two-sided ideal of ${\M} \rtimes {\A}$ and this sequence splits by $s \colon  {\A} \to {\M} \rtimes {\A}, s(a) = (0, a)$. Then, as in Example~\ref{action}~(iii), the above sequence induces another action of ${\A}$ on $\M$ given by
\begin{align*}
	{}^{a \cdot}m &= i^{-1}\big((0,a) (m,0)\big), \qquad  m^{\cdot a} = i^{-1}\big((m,0) (0,a)\big), \\
	{}^{a \ast}m & = i^{-1}\big[(0,a), (m,0)\big],  \quad \ \ m^{\ast a} = i^{-1}\big[(m,0), (0,a)\big],
\end{align*}
which actually matches the given one.

\subsection{Crossed modules}

\begin{definition}\label{def crossed module}
	A crossed module of  \awbs is a homomorphism of \awbs $\mu \colon  {\M} \to {\A}$ together with an action of $\A$ on $\M$ such that the following identities hold:
	\begin{itemize}
		\item[(CM1)]
		\begin{align*} 
			\mu(m^{ \cdot a})& = \mu(m) a, & \mu({^{a \cdot}} m) &= a \mu(m),\\
			\mu({}m^{ \ast a})& =  [\mu(m), a],  & \mu({}^{a \ast} m)& = [a, \mu(m)];
		\end{align*}
		\item[(CM2)]
		\begin{align*}
			{^{\mu(m)  \cdot}} m' &= m m'  \quad = m^{ \cdot {\mu(m')}},\\			
			{}^{\mu(m) \ast} m' &= [m, m']  =  m^{ \ast {\mu(m')}} 
			\end{align*}
	\end{itemize}
	for all $m, m' \in {\M}$, $a \in{\A}$.
\end{definition}

\begin{definition}
	A morphism of crossed modules $\left({\M} \xrightarrow{ \ \mu \ }  {\A} \right) \to \left({\M}' \xrightarrow{ \ \mu' \ }  {\A}' \right)$
	is a pair $(\alpha, \beta)$, where $\alpha \colon  {\M} \to {\M}'$ and $\beta \colon  {\A} \to {\A}'$ are homomorphisms of \awbs satisfying:
	\begin{itemize}
		\item[(a)] $\beta \circ \mu = \mu' \circ \alpha$.
		\item[(b)]
		\begin{align*}
		\alpha({^{a \cdot} m}) &= {^{\beta(a) \cdot}} \alpha(m), &  \alpha(m^{ \cdot a}) &=  \alpha(m)^{ \cdot {\beta(a)}} \ ,\\
		\alpha({}^{a \ast} m) &= {}^{\beta(a) \ast} \alpha(m), &  \alpha(m^{ \ast a}) &=   \alpha(m)^{ \ast {\beta(a)}}
		\end{align*}
	\end{itemize}
	for all $a \in {\A}$, $m \in {\M}$.
\end{definition}

It is clear that crossed modules of \awbs constitute a category, denoted by $\XAWB$.

\

The following lemma is an easy consequence of Definition~\ref{def crossed module}.
\begin{lemma} \label{cm}
	Let $\mu \colon  {\M} \to {\A}$ be a crossed module of \awbs. Then the following statements are satisfied:
	\begin{itemize}
		\item[(i)] $\Ker(\mu) \subseteq {\Z}(\M)$.
		\item[(ii)]  $\Ima(\mu)$ is a two-sided ideal of ${\A}$.
		\item[(iii)]  $\Ima(\mu)$ acts trivially on ${\Z}(\M)$, and so trivially on $\Ker(\mu)$. Hence $\Ker(\mu)$ inherits an action of ${\A}/\Ima(\mu)$ making $\Ker(\mu)$ a representation of the \awb  ${\A}/\Ima(\mu)$.
	\end{itemize}
\end{lemma}

\begin{example}
		\item[(i)] Let ${\A}$ be an \awb and ${\B}$ be a two-sided ideal of ${\A}$, then the inclusion ${\B} \hookrightarrow {\A}$  is a crossed module, where the action of ${\A}$ on ${\B}$ is given by the operations in ${\A}$ (see Example~\ref{action}~{(ii)}).  
	
\noindent	Conversely, if $\mu \colon  {\B} \to {\A}$ is a crossed module of \awbs and $\mu$ is an injective map, then ${\B}$ is isomorphic to a two-sided ideal of ${\A}$ by Lemma~\ref{cm} {(ii)}.
	
	\item[(ii)] For any representation ${\M}$ of an \awb  ${\A}$, the trivial map $0 \colon {\M} \to {\A}$ is a crossed module with the action of ${\A}$ on the abelian \awb  ${\M}$ described in Example~\ref{action}~{(i)}.
	
\noindent	Conversely, if $0 \colon  {\M} \to {\A}$ is a crossed module of \awbs, then ${\M}$ is necessarily an abelian \awb and the action of ${\A}$ on ${\M}$
	is equivalent to  ${\M}$ being a representation of ${\A}$.
	
	\item[(iii)] Any homomorphism of \awbs $\mu \colon  {\M} \to  {\A}$, with ${\M}$ abelian and $\Ima(\mu) \subseteq {\Z}({\A})$, provides a crossed module with ${\A}$ acting trivially on ${\M}$.
	
	\item[(iv)]  If  $0  \to {\N} \to {\M} \xrightarrow {\mu} {\A} \to 0$ is a central extension of \awbs,
	then $\mu$ is a crossed module with the induced action of ${\A}$ on ${\M}$ (see Example~\ref{action}~(v)).
\end{example}

\begin{proposition}  Let $\mu \colon  {\M} \to {\A}$ be a crossed module of \awbs. Then the maps
	\begin{itemize}
		\item[(i)] $(\mu, \id_{\A}) \colon  {\M} \rtimes {\A} \to {\A} \rtimes {\A}$,
		\item[(ii)] $(\id_{\M}, \mu) \colon  {\M} \rtimes {\M} \to {\M} \rtimes {\A}$,
		\item[(iii)]  $\varphi \colon  {\M} \rtimes {\A} \to {\M} \rtimes {\A}$ given by $\varphi(m, a) = (-m, \mu(m) + a)$,
	\end{itemize}
	are homomorphisms of \awbs.
\end{proposition}
\begin{proof}
	(i) is a direct consequence of equalities in (CM1) of Definition~\ref{def crossed module},  (ii) follows from equalities in  (CM2), whilst  (iii) requires both  (CM1) and  (CM2).
\end{proof}

	\begin{remark}
		The functors $I$ and $T$ given in Example \ref{ejemplos} (ii) and (iii) preserve actions and crossed modules in the sense of the following assertions:
		\begin{itemize}
			\item[(i)] Any action of an associative algebra $\A$ on  another associative algebra $\M$, $\A\times \M \to \M$, $(a,m)\mapsto a\cdot m$ and 
			$M\times A \to M$, $(m,a)\mapsto m\cdot a$ (see \cite{DIKL, DL}) defines an action of the \awb  $I(\A)$ on  $I(\M)$ (resp. of $T(\A)$  on $T(\M)$),  by letting
			\begin{align*}
				&{}^{a \cdot} m = a\cdot m, \  m^{\cdot a}= m\cdot a, \	 {}^{a \ast} m=0, \ m^{ \ast a}=0,\\
				\big(\text{resp.} \quad	&{}^{a \cdot} m = a\cdot m, \  m^{\cdot a}= m\cdot a, 	\ {}^{a \ast} m=a\cdot m - m\cdot a, \  m^{ \ast a}=m\cdot a - a\cdot m  \,	\big) 
			\end{align*}
			for all $a\in \A$ and $m\in \M$.
			\item[(ii)] If $\mu \colon \M\to \A$ is a crossed module of associative algebras (see again \cite{DIKL, DL}), then the homomorphisms of \awbs 
		$I(\mu) \colon I(\M)\to I(\A)$ and $T(\mu) \colon T(\M)\to T(\A)$, together with the actions of $I(\A)$ on  $I(\M)$ and of $T(\A)$ on $T(\M)$, are crossed modules of \awbs. 
		\end{itemize}
	\end{remark}

In \cite{CKL} we proved equivalence of crossed modules of \awbs with internal categories in the category of \awbs. Now we show their equivalence with $\text{cat}^1$-\awbs.  The following definition of $\text{cat}^1$-\awb is given in complete analogy with Loday's original notion of $\text{cat}^1$-groups \cite{Lo82}.

\begin{definition} A $\text{cat}^1$-\awb  $({\re}, {\pe}, s, t)$ consists of an \awb  ${\re}$, together
with a subalgebra ${\pe}$ and two homomorphisms $s, t \colon {\re} \to {\pe}$ of \awbs 
satisfying the following conditions:
\begin{enumerate}
\item[(a)] $s{\mid}_{\pe} = t{\mid}_{\pe} = \id_{{\pe}}$.
\item[(b)] ${\Ker}(s) ~ {\Ker}(t) = 0 = {\Ker}(t) ~ {\Ker}(s)$.
\item[(c)] $[{\Ker}(s), {\Ker}(t)] = 0 = [{\Ker}(t), {\Ker}(s)]$.
\end{enumerate}
\end{definition}

\begin{definition}
A morphism of $\text{cat}^1$-\awbs $({\re}, {\pe}, s, t) \to ({\re}', {\pe}', s', t')$ is a homomorphism of
\awbs $f \colon {\re} \to {\re}'$ such that $f({\pe}) \subseteq {\pe}'$  and $s' \circ f = f{\mid}_{\pe} \circ s$, $t' \circ f = f{\mid}_{\pe} \circ  t$.
\end{definition}

We let ${\mathsf{cat^1}\!-\!{\AWB}}$ denote the category of $\text{cat}^1$-\awbs.
Then we have the following theorem.

\begin{theorem} \label{equiv}
The categories ${\mathsf {cat^1}\!-\!\AWB}$ and ${\XAWB}$ are equivalent.
\end{theorem}
\begin{proof}
To a given $\text{cat}^1$-\awb  $({\re}, {\pe}, s, t)$ we associate a crossed module $\mu = t{\mid}_{\M} \colon \M \to \pe$,
 where ${\M} = {\Ker}(s)$ and the action of ${\pe}$ on ${\M}$ is given by the operations in ${\re}$ (see
	Example~\ref{action}~(ii)). It is easy to see that $\mu \colon {\M} \to {\pe}$ is a crossed module of \awbs and the assignment defines 
	a functor $\Phi \colon  {\mathsf {cat^1}\!-\!\AWB} \lra \XAWB$.

Conversely, let $\mu \colon {\M} \to {\pe}$ be a crossed module of \awbs, then the associated $\text{cat}^1$-\awb is given by
   $s, t \colon {\M} \rtimes {\pe} \to {\pe}$, where $ s(m,p)=p$, $t(m,p)= \mu(m)+p$, $m \in {\M}$, $p \in {\pe}$.
    It is straightforward to see that this assignment is functorial and provides a quasi-inverse functor for  $\Phi$.
\end{proof}

\section{Non-abelian tensor product of \awbs} \label{non-ab}

\begin{definition}
Let ${\M}$ and ${\N}$ be \awbs with mutual actions on each other. The actions are said to be compatible if
{\allowdisplaybreaks
\begin{align}\label{compatible}
	m^{\cdot ({^{m' \cdot}}n')} &= m({m'}^{\cdot n'}), &  m^{\cdot (n'{^{\cdot {m'}}})} &= m (^{n' \cdot}{m'}), \notag\\
	m^{\cdot({^{m' \ast}}n')} &= m({m'}^{\ast n'}), & m^{\cdot(^{n'^{\ast m'}})} &= m(^{n' \ast}{m'}),\notag\\
	m^{\ast (^{m' \cdot}n')} &= [m, m'^{\cdot n'}], & m^{\ast (n'^{\cdot m'})} &= [m, ^{n' \cdot} m'],\notag \\
	m^{\ast(^{m'\ast} n')} &= [m, m'^{\ast n'}], & m^{\ast(n'^{\ast m'})} &= [m, {^{n'\ast}} m'],\\
	^{(m^{\cdot n}) \cdot }n' &= (^{m \cdot} n)n', &  ^{(^{n \cdot} m) \cdot }n' &= (n^{ \cdot m})n' ,\notag\\
	^{ {(m^{\ast n}) \cdot}} n' &= ({^{m \ast}} n)n', & {^{(^{n \ast} m)\cdot}} n' &= (n^{\ast m})n',\notag\\
	{^{(m^{\cdot n}){\ast}}} n' &=  [{^{m \cdot}}n,n'], & {^{(^{n \cdot}m){\ast}}} n' &= [n^{\cdot m},n'], \notag\\
	{^{(m^{\ast n})\ast}} n' &= [{^{m \ast}} n, n'], & {^{(^{n \ast} m)\ast}} n'  &= [{n^{\ast m}}, n'], \notag
\end{align}
}and moreover, another $16$ equations obtained  by exchanging the roles of elements of ${\M}$ and ${\N}$ in \eqref{compatible} are also valid.
\end{definition}

\begin{example} \label{compatible actions}\
\begin{enumerate}
\item[(a)] If ${\M}$ and ${\N}$ are two-sided ideals of an \awb  $\A$, then the mutual actions on each other considered in Example~\ref{action}~(ii) are compatible.

\item[(b)] Let $\mu \colon {\M} \to {\pe}$ and $\nu \colon {\N} \to {\pe}$ be two crossed modules of \awbs. Then the mutual actions of ${\M}$ on ${\N}$ via $\mu$ and of ${\N}$ on ${\M}$ via $\nu$ are compatible.
\end{enumerate}
\end{example}

Let ${\M}$ and ${\N}$ be \awbs with mutual compatible actions on each other. We denote by ${\M} \odot {\N}$ the vector space spanned by all symbols $m \odot n$, $n \odot m$ and by ${\M} \circledast {\N}$ the vector space spanned by all
symbols $m \circledast n$, $n \circledast m$, for $m\in \M$, $n\in \N$. Let ${\M} \boxtimes {\N}$ denotes the quotient of $ \left(  {\M} \odot {\N}\right)  \oplus \left( {\M} \circledast {\N}\right) $ by the following relations:
\begin{align}\label{non-abelian}
	\lambda (m \star n) &= (\lambda m) \star n = m \star (\lambda n),& & \notag\\
	(m + m') \star n &= m \star n + m' \star n, & m \star (n + n') &= m \star  n + m \star n',\notag\\
	m^{ \cdot n} \star {^{m' \cdot}}  n' &= {^{m \cdot}} n \star m'^{ \cdot  n'}, & m^{ \cdot n} \ \star {n'^{\cdot m'}} &= {^{m \cdot}} n \star {^{ n' \cdot }}{m'},\notag\\
	{^{n \cdot}}m \star {n'^{\cdot m'}}  &= n^{\cdot m}  \star ^{n' \cdot}m', & ^{n \cdot}m \star {^{m' \cdot}}  n' &= n^{ \cdot m}  \star m'^{ \cdot  n'},\notag\\
	m^{ \cdot n} \star {{}^{m' \ast}}  n' &= {^{m \cdot}} n \star m'^{ \ast  n'}, & m^{ \cdot n} \star {n'^{\ast  m'}} &= {^{m \cdot}} n \star {{}^{n' \ast}}{m'},\notag\\
	{^{n \cdot}}m \star {{}^{m' \ast}} n'  &= n^{\cdot m}  \star m'^{\ast n'}, & ^{n \cdot}m \star n'^{{\ast  m'}} &= n^{ \cdot m}  \star {{}^{n' \ast }}m',\notag\\
	m^{ \ast n} \star {^{m' \cdot}}  n' &= {^{m \ast}} n \star m'^{\cdot  n'}, & m^{ \ast n} \star {n'^{\cdot  m'}} &= {^{m \ast}} n \star {^{n' \cdot}}{m'},\notag\\
	{^{n \ast}}m \star {^{m' \cdot }}n'  &= n^{\ast m} \star m'^{\cdot n'}, & ^{n \ast}m \star n'^{{\cdot  m'}} &= n^{ \ast m}  \star {^{n' \cdot }}m',\\
	m^{ \ast n} \star {{}^{m' \ast}}  n' &= {{}^{m \ast}} n \star m'^{\ast  n'}, & m^{ \ast n} \star {n'^{\ast  m'}} &= {^{m \ast}} n \star {^{n' \ast}}{m'},\notag\\
	{^{n \ast}}m \star {^{m' \ast }}n'  &= n^{\ast m}  \star m'^{\ast n'}, & ^{n \ast}m \star n'^{{\ast  m'}} &= n^{ \ast m}  \star {^{n' \ast }}m',\notag\\
	(m_1 m_2) \odot n &= m_1  \odot \left(^{m_2 \cdot}n \right), &  n \odot (m_1 m_2) &= ({n}^{\cdot m_1}) \odot m_2, \notag \\
	(m_1 m_2) \circledast n &= m_1 \odot ({^{m_2 \ast}}n) + ({^{m_1 \ast}} n) \odot m_2, &\notag \\
	{^{m_1 \cdot}} n \circledast m_2 &= m_1 \odot n^{\ast m_2} + [m_1, m_2] \odot n, &  &  \notag \\
	{n^{\cdot m_1}} \circledast m_2 &=  n^{\ast m_2} \odot m_1 + n \odot [m_1, m_2], &  &  \notag \\
	{^{m_1 \cdot}}n \odot m_2&= m_1 \odot n^{\cdot m_2}, &\notag
\end{align}
and another $25$  relations obtained by exchanging the roles of elements of ${\M}$ and ${\N}$ in \eqref{non-abelian}, where  the symbol $\star$  stands for both $\odot$ or $\circledast$. 

\begin{proposition} \label{structure}
The  vector space ${\M} \boxtimes {\N}$ endowed with the product and bracket operations given on the generators by
{\allowdisplaybreaks
\begin{align*}
	(m \odot n)   (m' \odot n') &= (m^{ \cdot n}) \odot (^{m' \cdot}  n'),&  (m \odot n)   (n' \odot m') &= (m^{ \cdot n}) \odot (n'^{ \cdot m'}),\\	
	(n \odot m)  (m' \odot n') &= ({^{n \cdot}} m) \odot ({^{m' \cdot}} n'), & (n \odot m)  (n' \odot m') &= ({^{n \cdot}} m) \odot ({n'}^{ \cdot {m'}}), \\
	(m \odot n)  (m' \circledast n') &= (m^{ \cdot n}) \odot ({^{m' \ast}} n'),&  (m \odot n)  (n' \circledast m') &= (m^{ \cdot n}) \odot (n'^{ \ast m'}),\\	
	(n \odot m)  (m' \circledast n') &= ({^{n \cdot}} m) \odot ({m'}^{ \ast n'}), & (n \odot m)  (n' \circledast m') &= ({^{n \cdot}} m) \odot ({n'}^{ \ast m'}), \\
		(m \circledast n)  (m' \odot n') &= (m^{ \ast n}) \odot ({^{m' \cdot}} n'),&  (m \circledast n)  (n' \odot m') &= (m^{ \ast n}) \odot (n'^{ \cdot m'}),\\
		(n \circledast m)  (m' \odot n') &= ({^{n \ast }}m) \odot ({^{m' \cdot}} n'), & (n \circledast m)  (n' \odot m') &= ({^{n \ast}} m) \odot ({n'}^{ \cdot m'}), \\
		(m \circledast  n)  (m' \circledast  n') &= (m^{ \ast n}) \odot ({^{m' \ast}} n'),&  (m \circledast  n) (n' \circledast  m') &= (m^{ \ast n}) \odot (n'^{ \ast m'}),\\	
	(n \circledast  m)  (m' \circledast  n') &= ({^{n \ast}} m) \odot ({^{m' \ast}} n'), & (n \circledast  m)  (n' \circledast  m') &= ({^{n \ast}} m) \odot ({n'}^{ \ast m'}), \\
		[m \odot n, m' \odot n'] &= (m^{ \cdot n}) \circledast ({^{m' \cdot}} n'), & [m \odot n, n' \odot m'] &= (m^{ \cdot n}) \circledast (n'^{ \cdot m'}),\\
		[n \odot m, m' \odot n'] &= ({^{n \cdot}} m) \circledast ({^{m' \cdot}} n'), &  [n \odot m, n' \odot m'] &= ({^{n \cdot}} m) \circledast ({n'}^{ \cdot m'}),\\
		[m \odot n, m' \circledast n'] &= (m^{ \cdot n}) \circledast ({^{m' \ast}} n'), & [m \odot n, n' \circledast m'] &= (m^{ \cdot n}) \circledast (n'^{ \ast m'}),\\
		[n \odot m, m' \circledast n'] &= ({^{n \cdot}} m) \circledast ({^{m' \ast}} n'), &  [n \odot m, n' \circledast m'] &= ({^{n \cdot}} m) \circledast ({n'}^{ \ast m'}),\\
		[m \circledast n, m' \odot n'] &= (m^{ \ast n}) \circledast ({^{m' \cdot}} n'), & [m \circledast n, n' \odot m'] &= (m^{ \ast n}) \circledast (n'^{ \cdot m'}),\\	
	[n \circledast m, m' \odot n'] &= ({^{n \ast}} m) \circledast ({^{m' \cdot}} n'), &  [n \circledast m, n' \odot m'] &= ({^{n \ast}} m) \circledast ({n'}^{ \cdot m'}),\\
		[m \circledast n, m' \circledast n'] &= (m^{ \ast n}) \circledast (^{m' \ast} n'), & [m \circledast n, n' \circledast m'] &= (m^{ \ast n}) \circledast (n'^{ \ast m'}),\\
		[n \circledast m, m' \circledast n'] &= (^{n \ast} m) \circledast ({}^{m' \ast} n'), &  [n \circledast m, n' \circledast m'] &= ({^{n \ast}} m) \circledast ({n'}^{ \ast m'}),
\end{align*}
}
has the structure of an \awb.
\end{proposition}
\begin{proof}
Straightforward calculations show that, under the conditions of compatible actions~\eqref{compatible}, by using the relations in \eqref{non-abelian}, the described operations  on ${\M} \boxtimes {\N}$ satisfy the fundamental identity \eqref{awbequa}.
\end{proof}

\begin{definition}
The structure of \awb  on  ${\M} \boxtimes {\N}$ provided by Proposition~\ref{structure} is called the non-abelian tensor product of the \awbs   ${\M}$ and ${\N}$.
\end{definition}

In particular, if the actions are trivial, the non-abelian tensor product can be described as follows.
	
	\begin{proposition}	\label{P:trivial_action}
If ${\M}$ and ${\N}$ are two \awbs with trivial actions on each other, then there is an isomorphism of abelian \awbs
	\[
	\M\boxtimes \N \cong \left( \M^{ab}\otimes_{\K} \N^{ab}\right) \oplus \left( \N^{ab}\otimes_{\K} \M^{ab}\right) \oplus \left(\M^{ab}\otimes_{\K} \N^{ab}\right)\oplus \left( \N^{ab}\otimes_{\K} \M^{ab}\right).
	\]
	\end{proposition}
\begin{proof}
	Equations in Proposition~\ref{structure} show us easily that  $\M\boxtimes \N$ is abelian in the case of trivial actions.
	The defining relations \eqref{non-abelian} of the non-abelian tensor product say that the vector space $\M\boxtimes \N$ is the quotient of 
$\left( \M\otimes_{\K} \N\right) \oplus \left( \N\otimes_{\K} \M \right) \oplus \left( \M\otimes_{\K} \N\right) \oplus \left( \N\otimes_{\K} \M \right)$ by the relations 
\begin{align*}
0 &=	(m_1 m_2) \otimes n = [m_1, m_2] \otimes n \\ 
&= 	n \otimes (m_1 m_2) = n \otimes [m_1, m_2] \\
& = 	m \otimes (n_1 n_2) = m \otimes [n_1, n_2] \\
&=	(n_1 n_2) \otimes m = [n_1, n_2] \otimes m  
\end{align*}
for all $m,m_1,m_2\in \M$,  $n,n_1,n_2\in \N$. This provides the required isomorphism.
\end{proof}

The non-abelian tensor product of \awbs is functorial in the following sense: let $f \colon {\M} \to {\M}'$ and $g \colon {\N} \to {\N}'$ be homomorphisms of \awbs together with compatible mutual actions of ${\M}$ and ${\N}$, also ${\M}'$  and ${\N}'$ on each other such that $f$, $g$ preserve these actions, i.e.
\begin{align*}
	f(^{n \cdot} m)&= {{}^{g(n) \cdot} f(m)}, \ f(m^{\cdot n}) = f(m)^{\cdot g(n)}, \ \   f(^{n \ast} m) ={}^{g(n) \ast} f(m), \ f(^{n \ast} m) = {{}^{g(n) \ast} f(m)},\\
	g(^{m \cdot} n)&= {{}^{f(m) \cdot} g(n)},  \ g(n^{\cdot m}) = g(n)^{\cdot f(m)}, \quad   \  g(^{m \ast} n) = {}^{f(m) \ast} g(n),  \ \ g(^{m \ast} n) = {{}^{f(m) \ast} g(n)}.
\end{align*}
for all $m \in {\M}, n \in {\N}$, then there is a homomorphism of \awbs
\[f \boxtimes g \colon {\M} \boxtimes {\N}  \xrightarrow{\ \ }  {\M}' \boxtimes {\N}'\]
defined by
\begin{align*}
	\left( f \boxtimes g\right)  (m \odot n) &= f(m) \odot g(n),  \quad    \left( f \boxtimes g\right)(n \odot m) =g(n) \odot f(m),\\
	\left( f \boxtimes g\right) (m \circledast n) &= f(m) \circledast g(n),  \quad   \left( f \boxtimes g\right)(n \circledast m) =g(n) \circledast f(m).
\end{align*}

The non-abelian tensor product of \awbs has a kind of right-exactness property presented in the following theorem.

\begin{theorem}\label{T:right_exactness}
	Let $0\xrightarrow{} \M_1 \xrightarrow{f} \M_2 \xrightarrow{g} \M_3 \xrightarrow{} 0$ be a short exact sequence of \awbs. Let $\N$ be an \awb  together with compatible actions of $\N$ and $\M_i$ $(i=1,2,3)$ on each other and $f$, $g$ preserve these actions.
	Then there is an exact sequence of \awbs
	\[
	\M_1\boxtimes \N \xrightarrow{\,f \boxtimes \id_{\N} \,}  \M_2\boxtimes \N \xrightarrow{\,g\boxtimes \id_{\N}\,} \M_3\boxtimes \N \xrightarrow{} 0.
	\]
\end{theorem}
\begin{proof}
	It is clear that the composition $\left( g\boxtimes \id_{\N}\right) \left( f \boxtimes \id_{\N}\right)$ is the trivial map, i.e. $\Ima\left( f \boxtimes \id_{\N}\right)\subseteq \Ker\left( f \boxtimes \id_{\N}\right)$ and at the same time $f \boxtimes \id_{\N}$ is an epimorphism.
	
	$\Ima\left( f \boxtimes \id_{\N}\right)$ is generated by the elements of the form $f(m_1) \odot n$, $n \odot f(m_1)$, $f(m_1) \circledast n$ and  $n \circledast f(m_1)$, for all $m_1\in \M_1$, $n\in \N$. Since $f$ preserves actions of $\N$, by the relations given in Proposition~\ref{structure}, it is easily verified that  $\Ima\left( f \boxtimes \id_{\N}\right)$ is a two-sided ideal of $\M_2\boxtimes \N$. For instance, taking a  generator of the form  $m_2 \odot n'$ in $\M_2\boxtimes \N$ we have
	\begin{align*}
		&	\left( f(m_1) \odot n\right)   (m_2 \odot n') = f(m_1)^{ \cdot n} \odot {}^{m_2 \cdot} n'= f(m_1^{ \cdot n}) \odot {}^{m_2 \cdot} n' \in \Ima\left( f \boxtimes \id_{\N}\right) ,\\
		&	\left( f(m_1) \circledast n\right)   (m_2 \odot n') = f(m_1)^{ \ast n} \odot {}^{m_2 \cdot} n'= f(m_1^{ \ast n}) \odot {}^{m_2 \cdot} n' \in \Ima\left( f \boxtimes \id_{\N}\right).
	\end{align*}
	
	Then, there is a homomorphism of \awbs 
	\[
	\alpha \colon \left( \M_2\boxtimes \N \right) / \Ima\left( f \boxtimes \id_{\N}\right) \lra \left( \M_3\boxtimes \N \right)
	\]
	induced by $g \boxtimes \id_{\N}$, that is, defined on generators by
	\begin{align*}
		&	\alpha \left( \overline{ m_2 \odot n} \right) =  g(m_2) \odot n, \quad  \alpha \left( \overline{ n \odot m_2} \right) =  n \odot g(m_2), \\
		&	\alpha \left( \overline{ m_2 \circledast n} \right) =  g(m_2) \circledast n, \quad \alpha \left( \overline{ n \circledast m_2} \right) =  n \circledast g(m_2). 
	\end{align*}
	where the overdrawn generator denotes the coset of the corresponding element.
	On the other hand, we have well-defined homomorphism of \awbs
	\[
	\alpha' \colon \left( \M_3\boxtimes \N \right) \lra \left( \M_2\boxtimes \N \right) / \Ima\left( f \boxtimes \id_{\N}\right)  
	\]
	given on generators by 
	\begin{align*}
		&	\alpha' \left(  m_3 \odot n \right) =  \left( \overline{ m_2 \odot n} \right), \quad 
		\alpha' \left( \overline{ n \odot m_3} \right) =   \overline{ n \odot m_2} , \\
		&	\alpha' \left( { m_3 \circledast n} \right) =  \overline{ m_2 \circledast n}, \quad  \ \quad
		\alpha \left( { n \circledast m_3} \right) = \overline{ n \circledast m_2}, 
	\end{align*}
	where $m_2 \in \M_2$ is any element  such that $g(m_2) = m_3$. Obviously $\alpha$ and $\alpha'$ are inverse to each other, i.e. $\alpha$ is an isomorphism.
	Then the required exactness follows. 
\end{proof}

\begin{proposition}\label{maps_phi_psi}
Let ${\M}$ and ${\N}$ be \awbs with compatible actions on each other. 
\begin{enumerate}
\item[(a)] There  are homomorphisms of \awbs

$\begin{array}{ll} \ \ \ \ \  \psi_{\M} \colon {\M} \boxtimes {\N} \to {\M} \ \text{given by}, & \psi_{\M}(m \odot n) = m^{ \cdot n}, \ \psi_{\M}(n \odot m) = {^{n \cdot}} m,\\
 & \psi_{\M}(m \circledast n) = m^{ \ast n}, \ \psi_{\M}(n \circledast m) = {^{n \ast}} m ; \end{array}$

$\begin{array}{ll} \text{and} \  \psi_{\N} \colon {\M} \boxtimes {\N} \to {\N} \ \text{given by}, & \psi_{\N}(m \odot n) = {^{m \cdot}} n, \ \psi_{\N}(n \odot m) = {n}^{ \cdot m},\\
 & \psi_{\N}(m \circledast n) = {^{m \ast}} n, \ \psi_{\N}(n \circledast m) = {n}^{ \ast m} . \end{array}$

 \item[(b)]  There are actions of ${\M}$ and ${\N}$ on the non-abelian tensor product ${\M} \boxtimes {\N}$ given, for all $m, m' \in {\M}, n, n' \in {\N}$, by
 \begin{align*}
 	^{m \cdot}(m' \odot n') &= m \odot (^{m' \cdot} n'),&  {^{m \cdot}}(n' \odot m') &= m \odot (n'^{\cdot m'}),\\
 	(m' \odot n')^{\cdot m} &= ({^{m' \cdot}} n') \odot m,&  (n' \odot m')^{\cdot m} &= (n'^{\cdot m'}) \odot m,\\
 	{^{m \ast}}(m' \odot n') &= m \circledast (^{m' \cdot} n'),& {^{m \ast}}(n' \odot m') &= m \circledast (n'^{\cdot m'}),\\
 		(m' \odot n')^{\ast m} &= (^{m' \cdot} n') \circledast m,&  (n' \odot m')^{\ast m} &= (n'^{\cdot m'}) \circledast m,\\
 	^{m \cdot}(m' \circledast n') &= m \odot (^{m' \ast} n'),&  {^{m \cdot}}(n' \circledast m') &= m \odot (n'^{\ast m'}),\\
 	(m' \circledast n')^{\cdot m} &= (^{m' \ast} n') \odot m,&  (n' \circledast m')^{\cdot m} &= (n'^{\ast m'}) \odot m,\\
 	{^{m \ast}}(m' \circledast n') &= m \circledast (^{m' \ast} n'),& {^{m \ast}}(n' \circledast m') &= m \circledast (n'^{\ast m'}),\\
 	(m' \circledast n')^{\ast m} &= (^{m' \ast} n') \circledast m,&  (n' \circledast m')^{\ast m} &= (n'^{\ast m'}) \circledast m,
 \end{align*}
%
%
%
%
%
%
%
and 
\begin{align*}
	^{n \cdot}(m' \odot n') &= n \odot (m'^{ \cdot n'}),&  {^{n \cdot}}(n' \odot m') &= n \odot (^{n' \cdot} m'),\\
		(m' \odot n')^{\cdot n} &= (m'^{ \cdot n'}) \odot n,&  (n' \odot m')^{\cdot n} &= (^{n'\cdot} m') \odot n,\\
	{^{n \ast}}(m' \odot n') &= n \circledast (m'^{ \cdot n'}),& {^{n \ast}}(n' \odot m') &= n \circledast (^{n' \cdot} m'),\\
	(m' \odot n')^{\ast n} &= (m'^{ \cdot n'}) \circledast n,&  (n' \odot m')^{\ast n} &= (^{n' \cdot} m') \circledast n,\\
	^{n \cdot}(m' \circledast n') &= n \odot (m'^{ \ast n'}),&  {^{n \cdot}}(n' \circledast m') &= n \odot (^{n' \ast} m'),\\
	(m' \circledast n')^{\cdot n} &= (m'^{ \ast n'}) \odot n,&  (n' \circledast m')^{\cdot n} &= (^{n'\ast} m') \odot n,\\
	{^{n \ast}}(m' \circledast n') &= n \circledast (m'^{ \ast n'}),& {^{n \ast}}(n' \circledast m') &= n \circledast (^{n' \ast} m'),\\
	(m' \circledast n')^{\ast n} &= (m'^{ \ast n'}) \circledast n,&  (n' \circledast m')^{\ast n} &= (^{n' \ast} m') \circledast n.  \quad \  
\end{align*}
%
%
%
%
%
%
%

 \item[(c)] The homomorphisms $\psi_{\M}$ and $\psi_{\N}$ together with the actions described in the statement $(b)$ are crossed modules of \awbs.
\end{enumerate}
\end{proposition}
\begin{proof}
This is straightforward but tedious verification.
\end{proof}

\begin{theorem} \label{T:awb_uce}
If ${\A}$ is a perfect \awb, then $\psi_{\A} \colon {\A} \boxtimes {\A} \to {\A}$ is the universal central extension of ${\A}$.
\end{theorem}
\begin{proof}
Clearly $\psi_{\A} \colon {\A} \boxtimes {\A} \to {\A}$ is an epimorphism if $\A$ is perfect. Moreover, it is a crossed module of \awbs by 
Proposition~\ref{maps_phi_psi} (c). Then Lemma~\ref{cm} (i) says that it is a central extension.

To show the universal property,  consider any central extension $0 \to {\M} \to {\B} \xrightarrow{\phi}  {\A} \to 0$.  Since $\Ker(\phi) \subseteq \Z(\B)$, we get a well-defined homomorphism of \awbs  $\varphi \colon {\A} \boxtimes {\A} \to {\B}$ given on generators by  $\varphi(a \odot a') = b\,b'$ and  $\varphi(a \circledast a') = [b, b']$,  where $b$ and  $b'$ are any elements in  $\phi^{-1}(a)$ and $\phi^{-1}(a')$, respectively. Obviously $\phi\varphi=\psi_{\A} $. Moreover, since $\A$ is perfect, it follows by the equalities in Proposition~\ref{structure} that ${\A} \boxtimes {\A}$ is a perfect \awb  as well. Then \cite[Lemma 3.1]{Ca} implies that $\varphi$ is the unique homomorphism satisfying the required conditions.
\end{proof}

Bearing in mind that the universal central extension of a perfect \awb  is unique up to isomorphism, by \cite[Theorem 3.5]{Ca} we conclude that \[{\mathsf H_1^{\awb}}({\A}) \cong \Ker(\psi_{\A} \colon \A \boxtimes {\A} \twoheadrightarrow \A ).\]

Moreover, if $0 \to {\mathsf R} \to {\mathsf F} \xrightarrow{\rho}  {\A} \to 0$ is a free presentation of a perfect \awb ${\A}$, then its universal central extension is  
 \[0 \xrightarrow{\ \ \ } \frac{{\mathsf R} \cap [[ {\mathsf F},  {\mathsf F}]]}{[[ {\mathsf F},  {\mathsf R}]]} \xrightarrow{\ \ \ } \frac{[[ {\mathsf F},  {\mathsf F}]]}{[[ {\mathsf F},  {\mathsf R}]]} \xrightarrow{\ \rho^*}  {\A} \xrightarrow{\ \ \ } 0\]
(see \cite{Ca}), hence \[{\A} \boxtimes {\A} \cong \frac{[[ {\mathsf F},  {\mathsf F}]]}{[[ {\mathsf F},  {\mathsf R}]]}\]
due to the uniqueness (up to isomorphisms) of the universal central extension.

\begin{remark}
The article \cite{Ca} provides another description of the universal central extension of a perfect \awb  $\A$. 
 In particular, it is shown that, given an \awb  $\A$, the quotient $\frac{{\A}^{\otimes 2} \oplus {\A}^{\otimes 2}}{I_{\A}}$ has an \awb structure, where $I_{\A}$ is the image of the map $d_2 \colon {\A}^{\otimes 3}\oplus {\A}^{\otimes 3} \to {\A}^{\otimes 2}\oplus {\A}^{\otimes 2} $ in the homology chain complex  
$\left ( C_*^{\awb}({\A}) , d_* \right )$, that is, $I_{\A}$ is the subspace of ${\A}^{\otimes 2} \oplus {\A}^{\otimes 2}$ spanned by the elements of the form 
	\begin{align*}
	& (a_1 a_2) \otimes a_3 - a_1 \otimes (a_2 a_3), \\	
    & 	[a_1,a_2] \otimes a_3 + a_1 \otimes [a_2,a_3] - (a_1 a_2) \circ a_3, 
	\end{align*} 
	 for any $a_1, a_2, a_3  \in {\A}$. 
Moreover, if $\A$ is a perfect \awb, then it gives the construction of the universal central extension of ${\A}$. As a consequence, we have the following isomorphism 
of \awbs
\[ 
\frac{{\A}^{\otimes 2} \oplus {\A}^{\otimes 2}}{I_{\A}}   \xrightarrow{\ \cong }  {\A} \boxtimes {\A}, 
\]
given by $a_1 \otimes a_2 \mapsto a_1 \odot a_2$ and $a_1 \circ a_2  \mapsto a_1 \circledast a_2$.   
\end{remark}

\begin{proposition} \label{sequence}
If ${\M}$ is a two-sided ideal of an \awb ${\A}$, then there is the exact sequence of \awbs
\[
\left( {\M} \boxtimes {\A}\right)  \rtimes ({\A} \boxtimes {\M}) \xrightarrow{\sigma}  {\A} \boxtimes {\A} \xrightarrow{\tau}  {\A}/{\M} \boxtimes {\A}/{\M} \to 0. 
\]

\end{proposition}
\begin{proof}
The functorial property of the non-abelian tensor product applied to the projection ${\A} \twoheadrightarrow {\A}/{\M}$ induces the surjective homomorphism $\tau$, and applied to $\inc \colon \M \to \A$ and $\id \colon {\A} \to {\A}$ provides the homomorphisms $\sigma' \colon {\M} \boxtimes {\A} \to {\A} \boxtimes {\A}$ and $\sigma'' \colon {\A} \boxtimes {\M} \to {\A} \boxtimes {\A}$.

Define $\sigma(x,y) = \sigma'(x) + \sigma''(y)$, for all $x \in {\M} \boxtimes {\A}$, $y \in {\A} \boxtimes {\M}$. $\Ima(\sigma)$ is a two sided ideal of ${\A} \boxtimes {\A}$ spanned by the elements of the form $m \odot a$, $a \odot m$, $m \circledast a$, $a \circledast m$ for all $a \in {\A}$ and $m\in {\M}$.

By the identities in Proposition~\ref{structure} and the relations \eqref{non-abelian} of the non-abelian tensor product, $\tau$ induces a homomorphism of \awbs $\bar{\tau} \colon \frac{{\A} \boxtimes {\A}}{\Ima(\sigma)} \to {\A}/{\M} \boxtimes {\A}/{\M}$. Define $\tau' \colon {\A}/{\M} \boxtimes {\A}/{\M} \to \frac{{\A} \boxtimes {\A}}{\Ima(\sigma)}$ by $\tau'\big((a_1 + {\M}) \odot (a_2 + {\M})\big) = a_1 \odot a_2 + {\Ima}(\sigma)$, $\tau'\big((a_1 + {\M}) \circledast (a_2 + {\M})\big) = a_1 \circledast a_2 + {\Ima(\sigma)}$. It is easy to check that $\tau'$ is a well-defined homomorphism that is inverse to $\bar{\tau}$.
\end{proof}

\begin{theorem} \label{T:4_term}
Let $\M$ be a two-sided ideal of a perfect \awb ${\A}$. Then there is an exact sequence of vector spaces
\[{\Ker}({\M} \boxtimes {\A} \xrightarrow{\psi_{\M}}   {\M}) \to {\mathsf H_1^{\awb}}({\A}) \to {\mathsf H_1^{\awb}}({\A}/{\M}) \to \frac{\M}{[[{\A},{\M}]]} \to 0\]
\end{theorem}
\begin{proof}
Due to Proposition~\ref{sequence} there is the following commutative diagram of \awbs with exact rows

\[ \xymatrix{
&({\M} \boxtimes {\A}) \rtimes ({\A} \boxtimes {\M}) \ar[r]^{\qquad \quad \sigma} \ar[d]^{\psi}& {\A} \boxtimes {\A} \ar[r]^{\tau \quad} \ar[d]^{\psi_{\A}}& {\A}/{\M} \boxtimes {\A}/{\M} \ar[r] \ar[d]^{\psi_{{\A}/{\M}}} & 0\\
0 \ar[r] & {\M} \ar[r] & {\A} \ar[r]^{\pi} & {\A}/{\M} \ar[r] & 0
}\]
where
\begin{align*}
 &\psi\big(m_1 \odot a_1, a_2 \odot m_2\big) = m_1 a_1 + a_2 m_2, \quad &&\psi\big(m_1 \odot a_1, a_2 \circledast m_2\big) = m_1 a_1 + [a_2,  m_2],\\
 &\psi\big(m_1 \circledast a_1, a_2 \odot m_2\big)= [m_1, a_1] + a_2 m_2, \quad  &&\psi\big(m_1 \circledast a_1, a_2 \circledast m_2\big) = [m_1, a_1] + [a_2, m_2]. 
\end{align*}

The Snake lemma provides the exact sequence 
\[
\Ker(\psi) \to \Ker(\psi_{\A}) \to \Ker(\psi_{{\A}/{\M}}) \to {\Coker}(\psi) \to 0.
\]
By Theorem~\ref{T:awb_uce} $\Ker(\psi_{\A})={\mathsf H}_1^{\awb}(\A)$, and since $\A/\M$ is a perfect \awb as well, we also have $\Ker(\psi_{\A/\M}) = {\mathsf H}_1^{\awb}(\A/\M)$.
Obviously ${\Coker}(\psi) = \frac{\M}{[[{\A},{\M}]]}$. Then the fact that there is a surjective map $\Ker(\psi_{\M}) \to  {\Ker}(\psi)$ completes the proof.
\end{proof}



\section*{Acknowledgments} 
This work was supported by Agencia Estatal de Investigaci\'on (Spain), grant PID2020-115155GB-I00. Emzar Khmaladze was
supported by Shota Rustaveli National Science Foundation of Georgia, grant FR-23-271. He is grateful to the University of Santiago de Compostela for its hospitality. Manuel Ladra was also funded by Xunta de Galicia, grant ED431C 2023/31 (European FEDER support included, UE).

\end{document}